\documentclass[12pt]{amsart}

\usepackage{amssymb}
\usepackage{amscd}

\title{Rational points and Galois points for a plane curve over a finite field}
\author{Satoru Fukasawa}

\address{Department of Mathematical Sciences, Faculty of Science, Yamagata University, Kojirakawa-machi 1-4-12, Yamagata 990-8560, Japan}
\email{s.fukasawa@sci.kj.yamagata-u.ac.jp} 
\thanks{The author was partially supported by JSPS KAKENHI Grant Number 25800002.}

\newtheorem{theorem}{Theorem}

\newtheorem{lemma}{Lemma}
\newtheorem{fact}{Fact}
\newtheorem{problem}{Problem} 

\theoremstyle{definition}

\begin{document}
\begin{abstract} 
We study the relationship between rational points and Galois points for a plane curve over a finite field. 
It is known that the set of Galois points coincides with that of rational points of the projective plane if the curve is the Hermitian, Klein quartic or Ballico-Hefez curves. 
We propose a problem: {\it Does the converse hold true?} 
When the curve of genus zero or one has a rational point, we will have an affirmative answer. 
\end{abstract}
\maketitle

\section{Introduction}  
We study the relationship between rational points and Galois points for a plane curve over a finite field. 

We recall the definition of Galois point, which was given by H. Yoshihara in 1996 (\cite{miura-yoshihara, yoshihara}). 
Let $C \subset \Bbb P^2$ be an irreducible plane curve of degree $d \ge 4$ over an algebraically closed field $K$ of characteristic $p \ge 0$ and let $K(C)$ be its function field. 
A point $P \in \Bbb P^2$ is said to be Galois for $C$, if the function field extension $K(C)/\pi_P^*K(\Bbb P^1)$ induced by the projection $\pi_P: C \dashrightarrow \Bbb P^1$ from $P$ is Galois. 
We denote by $\Delta(C)$ the set of all Galois points on the projective plane. 

Let $C$ be a plane curve over a finite field $\Bbb F_{q_0}$ which is irreducible over the algebraic closure $\overline{\Bbb F}_{q_0}$. 
We consider Galois points over $\overline{\Bbb F}_{q_0}$. 
Fukasawa and Hasegawa \cite{fukasawa2, fukasawa-hasegawa} showed that $\sharp\Delta(C) <\infty$ except for certain explicit examples.  
Therefore, we assume that $\sharp\Delta(C) <\infty$ here. 
Then,
$$ \Delta(C) \subset \Bbb P^2(\Bbb F_q); \ \mbox{ for infinitely many} \ q \ge q_0. $$
{\it When does $\Delta(C)=\Bbb P^2(\Bbb F_q)$ hold?}
Summarizing the results of Homma \cite{homma} and Fukasawa \cite{fukasawa1, fukasawa4}, we have the following very interesting theorem. 

\begin{fact}[Homma, Fukasawa] \label{examples} 
\begin{itemize}
\item[(1)] For the Hermitian curve $H_{\sqrt{q}+1}$: 
$$X^{\sqrt{q}}Z+XZ^{\sqrt{q}}-Y^{\sqrt{q}+1}=0,$$ $\Delta(H_{\sqrt{q}+1})=\Bbb P^2(\Bbb F_{q})$. 
\item[(2)] For the Klein quartic curve $K_4$: 
$$ (X^2+XZ)^2+(X^2+XZ)(Y^2+YZ)+(Y^2+YZ)^2+Z^4=0$$
in $p=2$, $\Delta(K_4)=\Bbb P^2(\Bbb F_2)$. 
\item[(3)] For the Ballico-Hefez curve $B_{q+1}$, which is the image of the morphism 
$$ \Bbb P^1 \rightarrow \Bbb P^2; \ (s:t) \mapsto (s^{q+1}:(s+t)^{q+1}:t^{q+1}), $$
$\Delta(B_{q+1})=\Bbb P^2(\Bbb F_q)$. 
\end{itemize}
\end{fact} 

We propose the following problem: 

\begin{problem} 
Let $C$ be a plane curve over $\Bbb F_q$. 
Assume that $\Delta(C)=\Bbb P^2(\Bbb F_q)$. 
Then, is it true that $C$ is projectively equivalent to the Hermitian, Klein quartic or Ballico-Hefez curve?
\end{problem}

When $C$ is smooth, it is already known that the answer is affirmative (\cite{fukasawa3}). 
Therefore, we consider singular curves. 
Let $C_{\rm sm}$ be the smooth locus of $C$. 
When $C_{\rm sm}(\Bbb F_q) \ne \emptyset$ and the geometric genus of $C$ is zero or one, we have an affirmative answer, as follows. 

\begin{theorem} 
Assume that the geometric genus of $C$ is zero or one.  
Then, $C_{\rm sm}(\Bbb F_q) \ne \emptyset$ and $\Delta(C)=\Bbb P^2(\Bbb F_q)$ if and only if $C$ is projectively equivalent to the Ballico-Hefez curve $B_{q+1}$ (over $\Bbb F_q$). 
\end{theorem}  

See \cite{ballico-hefez, fhk} for other properties of the Ballico-Hefez curves.

\section{Preliminaries}
Let $\Bbb F_q$ be a finite field, let $p$ be the characteristic, and let $C$ be a plane curve of degree $d \ge 4$ defined over $\Bbb F_q$ which is irreducible over the algebraic closure $\overline{\Bbb F}_q$. 
We denote by $\pi: \hat{C} \rightarrow C$ the normalization. 
We can take $\hat{C}$ and $\pi$ which are defined over $\Bbb F_q$. 

Let $(X:Y:Z)$ be a system of homogeneous coordinates of the projective plane $\Bbb P^2$.  
We denote by $S(\Bbb F_q)$ the set of all $\Bbb F_q$-points of a subset $S \subset \Bbb P^2$. 
For distinct points $P, R$, $\overline{PR}$ means the line passing through them.  
If $P \in C_{\rm sm}$, $T_PC \subset \Bbb P^2$ is the (projective) tangent line at $P$. 
For a projective line $\ell \subset \Bbb P^2$ and a point $P \in C \cap \ell$, $I_P(C, \ell)$ means the intersection multiplicity of $C$ and $\ell$ at $P$.  
If $\hat{P} \in \hat{C}$ is a non-singular branch, i.e. there exists a line defined by $h=0$ with ${\rm ord}_{\hat{P}}\pi^*h=1$, then there exists a unique tangent line at $P=\pi(\hat{P})$ defined by $h_{\hat{P}}=0$ such that ${\rm ord}_{\hat{P}}\pi^*h_{\hat{P}} \ge 2$.  
Let $\pi_P: C \dashrightarrow \Bbb P^1; R \mapsto \overline{PR}$ be the projection from a point $P \in \Bbb P^2$. 
We write $\hat{\pi}_P=\pi_P \circ \pi$. 
We denote by $e_{\hat{R}}$ the ramification index of $\hat{\pi}_P$ at $\hat{R} \in \hat{C}$. 
If $R=\pi(\hat{R}) \in C_{\rm sm}$, then we denote $e_{\hat{R}}$ also by $e_R$.   
It is not difficult to check the following.  

\begin{fact} \label{index}
Let $P \in \Bbb P^2$ and let $\hat{R} \in \hat{C}$ with $\pi(\hat{R})=R \ne P$. 
Then for $\hat{\pi}_P$ we have the following. 
\begin{itemize}
\item[(1)] If $P \in C_{\rm sm}$, then $e_P=I_P(C, T_PC)-1$.  
\item[(2)] If $h$ is a linear polynomial defining $\overline{RP}$, then $e_{\hat{R}}={\rm ord}_{\hat{R}}\pi^*h$. 
In particular, if $R$ is smooth, then $e_R =I_R(C, \overline{PR})$.  
\end{itemize} 
\end{fact}

For a Galois covering $\theta:C \rightarrow C'$ between smooth curves, the following fact is useful (see \cite[III. 7.1, 7.2 and 8.2]{stichtenoth}). 
\begin{fact} \label{Galois covering} 
Let $\theta: C \rightarrow C'$ be a Galois covering with Galois group $G$. 
We denote by $G(P)$ the stabilizer subgroup of $P \in C$. 
Then, we have the following. 
\begin{itemize}
\item[(1)] If $\theta(P)=\theta(Q)$, then $e_P=e_Q$. 
\item[(2)] The order $|G(P)|$ of $G(P)$ is equal to $e_P$ for each point $P \in C$. 
\end{itemize} 
\end{fact}

We mention properties of Galois covering between rational curves.  
The following fact is a corollary of the classification of finite subgroups of ${\rm PGL}(2, K)$ (see, for example, \cite[Theorem 11.91]{hkt}). 

\begin{fact} \label{rational-ramification} 
Let $\theta: \Bbb P^1 \rightarrow \Bbb P^1$ be a Galois covering of degree $d \ge 3$ and let $d=q\ell$, where $q$ is a power of $p$ and $\ell$ is not divisible by $p$. 
Then we have the following. 
\begin{itemize} 
\item[(1)] If $q=1$ and $\theta$ is ramified at $P$ with $e_P=d$, then there exist a unique ramification point $Q \ne P$ with $e_Q=d$.  
\item[(2)] If $q>1$, $\ell\ge 2$ and $\theta$ is ramified at $P$ with $e_P=d$, then $\ell$ divides $q-1$ and there exists a point $Q \ne P$ such that $e_Q=\ell$. 
\end{itemize} 
\end{fact}

\section{Proof} 
Throughout in this section, we assume that $\Delta(C)=\Bbb P^2(\Bbb F_q)$. 

\begin{lemma} \label{total flex}
If $P \in C_{\rm sm}(\Bbb F_q)$, then $I_P(C, T_PC)=d$.  
\end{lemma} 

\begin{proof} 
Since $P$ is $\Bbb F_q$-rational, the tangent line $T_PC$ is defined over $\Bbb F_q$. 
Assume that $T_PC \cap C \setminus \{P\} \ne \emptyset$. 
Let $Q$ be such a point and let $\hat{Q} \in \hat{C}$ with $\pi(\hat{Q})=Q$. 
There exists a point $P' \in T_PC(\Bbb F_q) \setminus \{P, Q\}$, since $\sharp T_PC(\Bbb F_q) \ge 3$. 
Let $h$ be a linear polynomial defining $T_PC$ around $Q$. 
Considering the projection $\hat{\pi}_P$, by Facts \ref{index} and \ref{Galois covering}(1), ${\rm ord}_{\hat{Q}}\pi^*h=I_P(C, T_PC)-1$. 
Considering $\hat{\pi}_{P'}$, by Facts \ref{index}(2) and \ref{Galois covering}(1), we also have ${\rm ord}_{\hat{Q}}\pi^*h=I_P(C, T_PC)$. 
This is a contradiction. 
Therefore, we have $C \cap T_PC=\{P\}$. 
\end{proof}

\begin{lemma} \label{singularity}
Assume that $C_{\rm sm}(\Bbb F_q) \ne \emptyset$. 
If $Q \in C(\Bbb F_q)$, then $\sharp \pi^{-1}(Q)=m(Q)$, where $m(Q)$ is the multiplicity at $Q$. 
\end{lemma}

\begin{proof}
Let $P \in C_{\rm sm}(\Bbb F_q)$. 
Then, the line $\overline{PQ}$ is $\Bbb F_q$-rational. 
By Lemma \ref{total flex}, $I_P(C, \overline{PQ})=1$. 
There exists an $\Bbb F_q$-point $P'$ such that $P' \in \overline{PQ} \setminus \{P, Q\}$. 
For the projection from $P'$, by Facts \ref{index}(2) and \ref{Galois covering}(1), each point in $\pi^{-1}(Q)$ is not a ramification point. 
Therefore, we have the conclusion. 
\end{proof}

\begin{lemma} 
If $C_{\rm sm}(\Bbb F_q) \ne \emptyset$, then $\hat{C}$ is not elliptic. 
\end{lemma} 

\begin{proof}
Assume by contradiction that $E=\hat{C}$ is elliptic. 
Let $P \in C_{\rm sm}(\Bbb F_q)$. 
Then, by Lemma \ref{total flex}, there are two points $P', P''\in (T_PC \cap (\Bbb P^2 \setminus C))(\Bbb F_q)$. 
By Facts \ref{index} and \ref{Galois covering}, any automorphism of $E$ given by the Galois groups at $P, P', P''$ fixes the point $P$.  
It is known that the order of the automorphism group ${\rm Aut}(E, P)$ of an elliptic curve $E$ fixing a point $P$ divides $24$, and the order is $24$ (resp. $12$) only if $p=2$ (resp. $p=3$) (see \cite[III, Theorem 10.1]{silverman}). 
Since the orders of Galois groups at points $P, P'$ are $d-1, d$ respectively, $d(d-1)$ divides $24$.   
Then, we have $d=4$ and $p=2$ or $3$.

Let $G_{P'}, G_{P''}$ be the Galois groups at $P', P''$ respectively. 
According to \cite[Lemma 7]{fukasawa2},  $G_{P'} \cap G_{P''} =\{1\}$. 
Then, $|{\rm Aut}(E, P)| \ge 4 \times 4=16$.
Since $|{\rm Aut}(E, P)| \le 12$ if $p=3$, we have $p=2$ and $|{\rm Aut}(E, P)|=24$. 
Furthermore, $G_{P'}$ and $G_{P''}$ are contained in different Sylow $2$-groups of ${\rm Aut}(E, P)$ (whose order is eight).  
This is a contradiction to the fact that a Sylow $2$-group is unique, since it is known that a Sylow $2$-group of ${\rm Aut}(E, P)$ is a normal subgroup (\cite[Exercise A.1.(b)]{silverman}). 
\end{proof}

Hereafter, we assume that $\hat{C}=\Bbb P^1$ and $C_{\rm sm}(\Bbb F_q) \ne \emptyset$. 

\begin{lemma} \label{condition on d}
The characteristic $p$ does not divide $d$. 
\end{lemma}

\begin{proof} 
Let $P \in C_{\rm sm}(\Bbb F_q)$. 
Then, the tangent line $T_PC$ is $\Bbb F_q$-rational. 
Assume that $p$ divides $d$. 
Then $d-1$ is not divisible by $p$.  
By Lemma \ref{total flex}, Facts \ref{index}(1) and \ref{rational-ramification}(1), for $\hat{\pi}_P$, there exists a point $\hat{Q} \in \hat{C}$ such that the ramification index at $\hat{Q}$ is equal to $d-1$. 
For a suitable system of coordinates, we may assume that $T_PC$ is given by $X=0$ and $P=(0:0:1)$. 
We denote by $\pi=(f(s, t):g(s, t):h(s, t))$, where $f, g, h \in \Bbb F_q[s,t]$ are homogeneous polynomials of degree $d$. 
Since $C \cap \{X=0\}=\{P\}$, $f(s,1)=0$ has a unique solution $s=\alpha \in \Bbb F_q$. 
Therefore, $f(s, 1)=a(s-\alpha)^d$ for some $a \in \Bbb F_q$. 
For a suitable system of coordinates, we can take $f(s,1)=s^d$. 
The projection $\hat{\pi}_P$ is given by $(1:g(1, t))$. 
Since $g'(1, t)=0$ has a unique solution $t=\beta \in \overline{\Bbb F}_q$ corresponding to $\hat{Q}$,  $g'(1, t)=b(t-\beta)^{d-2}$ for some $b \in \Bbb F_q\setminus\{0\}$. 
Since $g' \in \Bbb F_q[t]$, $\beta^q=\beta$ and hence, $\hat{Q}$ is $\Bbb F_q$-rational. 
Then, $Q=\pi(\hat{Q})$ is Galois and $\pi^{-1}(C \cap \overline{PQ})=\{P, \hat{Q}\}$. 
By Lemma \ref{singularity}, $1=\sharp \pi^{-1}(Q)=m(Q)$ and hence, $Q \in C_{\rm sm}$. 
For $\hat{\pi}_Q$, $e_P=1$ and $e_{\hat{Q}}=d-2$. 
This is a contradiction to Fact \ref{Galois covering}(1). 
\end{proof}

\begin{lemma} \label{rational point}
We have $\pi(\Bbb P^1(\Bbb F_q)) \subset C_{\rm sm}(\Bbb F_q)$. 
\end{lemma} 

\begin{proof} 
Let $\hat{P}_1, \hat{P}_2 \in \Bbb P^2(\Bbb F_q)$ and $\pi(\hat{P}_1) \in C_{\rm sm}$. 
By Lemma \ref{singularity}, $\hat{P}_2$ is a non-singular branch and we can define the tangent line $T_{\hat{P}_i}C$ for $i=1, 2$. 
Then the intersection point $Q$ given by $T_{\hat{P}_1}C \cap T_{\hat{P}_2}C$ is $\Bbb F_q$-rational. 
Since $\hat{P}_1$ is a total ramification point for the Galois covering $\hat{\pi}_Q$ and $p$ does not divide $d$ by Lemma \ref{condition on d}, $\hat{P}_2$ must be also a total ramification point due to Fact \ref{rational-ramification}(1). 
We have $1=\sharp\pi^{-1}(\pi(\hat{P}_2))=m(\pi(\hat{P}_2))$. 
Hence, $\pi(\hat{P}_2)$ is smooth. 
\end{proof}

\begin{lemma} \label{distribution}
There exist $P_1, P_2, P_3 \in C_{\rm sm}(\Bbb F_q)$ and $Q_1, Q_2, Q_3 \in (\Bbb P^2 \setminus C)(\Bbb F_q)$ such that points $Q_1, Q_2, Q_3$ are not collinear and $T_{P_i}C=\overline{Q_jQ_k}$ for each $i, j, k$ with $\{i, j, k\}=\{1, 2, 3\}$.  
\end{lemma}

\begin{proof}
By Lemma \ref{rational point}, $\sharp C_{\rm sm}(\Bbb F_q) \ge q+1 \ge 3$. 
We take three distinct points $P_1, P_2, P_3$ and intersection points $Q_i$ given by $T_{P_j}C \cap T_{P_k}$ for $\{i,j, k\}=\{1, 2, 3\}$. 
By considering the sum of the ramification indices of the Galois coverings $\Bbb P^1 \rightarrow \Bbb P^1$ with two total ramification points, $Q_i \ne Q_j$ if $i \ne j$. 
We have $T_{P_i}C=\overline{Q_jQ_k}$ for each $\{i,j,k\}=\{1, 2, 3\}$ and $Q_1, Q_2, Q_3$ are not collinear. 
\end{proof}

\begin{proof}[Proof of Theorem]
By Lemma \ref{distribution}, for a suitable system of coordinates, we may assume that $T_{P_1}C, T_{P_2}C, T_{P_3}C$ are defined by $X=0$, $Y=0$, $Z=0$ respectively. 
We denote by $\pi=(f(s, t):g(s, t):h(s, t))$, where $f, g, h \in \Bbb F_q[s,t]$ are homogeneous polynomials of degree $d$. 
Since $C \cap \{X=0\}=\{P_1\}$, $f(s,1)=0$ has a unique solution $s=\alpha \in \Bbb F_q$. 
Therefore, $f(s, 1)=a(s-\alpha)^d$ for some $a \in \Bbb F_q$. 
For a suitable system of coordinates, we can take $f(s,1)=s^d$. 
Similarly, we can take $h(1, t)=t^d$ and $g(1, t)=(1+t)^d$. 
Therefore, $\pi$ is represented by $$(s:t) \mapsto (s^d:(s+t)^d:t^d).$$  

Now $P_3=\pi(1:0)=(1:1:0)$.
We consider the projection $\pi_{P_3}$. 
Since $\pi_{P_3}$ is represented by $(Y-X:Z)$, we have 
$$\hat{\pi}_{P_3}(s:1)=((s+1)^d-s^d:1). $$
Let $f(s)=(s+1)^d-s^d$. 
Assume that $p$ does not divide $d-1$. 
By Lemma \ref{total flex}, Facts \ref{index}(1) and \ref{rational-ramification}(1), $\pi_{P_3}$ has exactly two total ramification points and $f'(s)=0$ has a unique solution. 
However, $f'(s)=d(s+1)^{d-1}-ds^{d-1}=0$ implies $((1+s)/s)^{d-1}=1$. 
Then, we have $d-2$ solutions. 
This is a contradiction.
Therefore, $d-1$ is divisible by $p$. 

Let $d-1=q'\ell$, where $q'$ is a power of $p$ and $\ell$ is not divisible by $p$. 
Then, $\ell$ divides $q'-1$ by Fact \ref{rational-ramification}(2). 
Since the rank of the matrix
$$\left(\begin{array}{ccc}
1 & (t+1)^d & t^d \\
0 & d(t+1)^{d-1} & dt^{d-1} 
\end{array}\right) 
\sim
\left(\begin{array}{ccc}
1 & (t+1)^{d-1} & 0 \\
0 & (t+1)^{d-1} & t^{d-1} 
\end{array}\right) $$
is always two, for each point $\hat{P} \in \hat{C}$, we can define the tangent line $T_{\hat{P}}C$ which is given by 
$$ h_{\hat{P}}(X, Y, Z):=t^{d-1}(t+1)^{d-1}X-t^{d-1}Y+(t+1)^{d-1}Z=0. $$
Let $\hat{P}=(1:t_0)$ and let $u=t-t_0$. 
Since 
\begin{eqnarray*}
(1+t)^d & =& ((1+t_0)+u)^{q'\ell}((1+t_0)+u) \\
&=& ((1+t_0)^{q'}+u^{q'})^\ell((1+t_0)+u) \\
&=& (1+t_0)^d+(1+t_0)^{d-1}u+(\mbox{sum of terms of deg.} \ge q' \mbox{ on } u)
\end{eqnarray*} 
and 
$$ t^d=t_0^d+t_0^{d-1}u+(\mbox{sum of terms of deg.} \ge q' \mbox{ on } u), $$
we have
$$ \pi^*h_{\hat{P}}(1, y, z)=(\mbox{sum of terms of deg.} \ge q' \mbox{ on } u). $$
Therefore, ${\rm ord}_{\hat{P}}\pi^*h_{\hat{P}} \ge q'$ for each point $\hat{P} \in \hat{C}$, where $h_{\hat{P}}$ is a linear polynomial defining $T_{\hat{P}}C$.  
If $\ell \ge 2$, by Fact \ref{rational-ramification}(2), there exists a point $\hat{Q} \in \hat{C}$ such that ${\rm ord}_{\hat{Q}}\pi^*h_{\hat{Q}}=\ell$. 
Then, we have $q'>\ell \ge q'$. 
This is a contradiction.
We have $\ell=1$ and $q'=d-1$.  
By Fact \ref{examples}(3), $q'=q$. 

The if-part is nothing but Fact \ref{examples}(3). 
\end{proof}

\end{document}